\documentclass[12pt]{article}
\usepackage{amsmath,amssymb,amsthm,fullpage, hyperref, comment}
\usepackage[left=1in, right=1in, top=1.2in, bottom=1.2in]{geometry}
\usepackage{enumitem}
\usepackage{kpfonts}

\newtheorem{theorem}{Theorem}     
     
\newtheorem{qn}{Question}     
\newtheorem{corollary}{Corollary}

\theoremstyle{definition}
\newtheorem{definition}{Definition}

\newtheorem{example}{Example}

\newtheorem{remark}{Remark}

\title{\textbf{Random Finite Sum sets and Product Sets in Subsets of the Natural Numbers}}

\author{Sukrit Chakraborty\thanks{Department of Mathematics, Achhruram Memorial College, Jhalda, Purulia 723202, West Bengal, India. Email: sukritpapai@gmail.com} \and 
Sayan Goswami\thanks{Department of Mathematics, Ramakrishna Mission Vivekananda Educational and Research Institute, Belur Math, Howrah, West Bengal-711202, India. Email: sayan92m@gmail.com}\and 
Sourav Kanti Patra\thanks{Department of Mathematics, Kishori Sinha Mahila College, Q976+WXP, Aurangabad, Bihar 824101, India. Email: souravkantipatra@gmail.com}}

\date{\today}

\begin{document}
	
	\maketitle

	\begin{abstract}
		We investigate the occurrence of additive and multiplicative structures in random subsets of the natural numbers. 
		Specifically, for a Bernoulli random subset of $\mathbb{N}$ where each integer is included independently with probability $p\in(0,1)$, we prove that almost surely such a set contains finite sumsets (FS-sets) and finite product sets (FP-sets) of every finite length. Then we prove any Bernoulli random subset of $\mathbb{N}$ contains a pattern of the form $\{x,y,x+y,xy\},$ giving a random solution of the Hindman conjecture.

In addition, we establish a novel connection between Hindman’s partition theorem and the central limit theorem, providing a probabilistic perspective on the asymptotic Gaussian behavior of monochromatic finite sums and products. These results can be interpreted as probabilistic analogues of finite-dimensional versions of Hindman's theorem. 
		Applications, implications, and open questions related to infinite FS-sets and FP-sets are discussed.
	\end{abstract}



	\section{Introduction}
    Ramsey theory is built on the principle that large combinatorial structures inevitably contain highly regular configurations. A central challenge, however, is to understand how these deterministic guarantees behave when randomness is introduced. Just as random analogues of van der Waerden’s and Szemerédi’s theorems have reshaped our understanding of arithmetic structure in random sets, we investigate Hindman’s theorem in the probabilistic setting. Our results show that the hallmark finite-sum and finite-product configurations persist with high probability and exhibit central limit behavior, revealing a new intersection of Ramsey theory and probabilistic combinatorics.
    
	The study of additive and multiplicative structures in subsets of natural numbers has a long and rich history, ranging from classical results in combinatorics to modern developments in probabilistic and additive number theory \cite{Nathanson1996,TaoVu2006,AlonSpencer2016,Bollobas2001}. A central theme is understanding how structured configurations, such as arithmetic progressions, finite sumsets, and product sets, appear either deterministically or in a random setting \cite{Gowers2001,Komlos1996}.
	
	Hindman's theorem \cite{Hindman1974,Hindman1979,Baumgartner1974} is a cornerstone of additive Ramsey theory. It asserts that for any finite coloring of the natural numbers, there exists an infinite sequence whose finite sumset is monochromatic. This profound result guarantees the existence of highly structured sets under arbitrary colorings, highlighting the inevitability of additive combinatorial structure. Extensions and related results were further developed by Bergelson and Hindman \cite{Bergelson2001,Bergelson2016}, exploring partition-regular and polynomial configurations. Analogous multiplicative phenomena have been considered in the context of product sets, although the literature is less extensive.
	
	In probabilistic combinatorics, random subsets of $\mathbb{N}$ provide a natural framework to study the typical presence of additive and multiplicative structures. Classical works, such as those by Erd\H{o}s and R\'enyi \cite{ErdosRenyi1960}, established that almost every random subset contains various prescribed finite patterns, including arithmetic progressions. More recent probabilistic and sparse random set results were developed by Conlon and Gowers \cite{ConlonGowers2016} and R\"odl and Ruci\'nski \cite{RodlRucinski1995}, but explicit statements regarding finite sumsets (FS-sets) and finite product sets (FP-sets) in random subsets have not been widely formalized in the literature.

	In this work, we establish the following results, rigorously proving that almost every random subset of the natural numbers contains finite FS-sets and FP-sets of every finite length. These results can be interpreted as probabilistic analogues of the finite-dimensional versions of Hindman's theorem:
	
	\begin{itemize}
		\item \textbf{Theorem A (Random Finite Sumsets)}: In a random subset of $\mathbb{N}$ formed by including each integer independently with probability $p \in (0,1)$, almost surely for every $L \ge 1$ there exists a finite sumset of length $L$ contained in the set (see Theorem~\ref{thm:FS} for precise statement and proof).
		\item \textbf{Theorem B (Random Finite Product Sets)}: Similarly, in the same random model, almost surely for every $L \ge 1$ there exists a finite product set of length $L$ contained in the set (see Theorem~\ref{thm:FP} for precise statement and proof).
	\end{itemize}
	
	Beyond these foundational results, we uncover a novel connection between Hindman's theorem and the central limit theorem, offering a probabilistic perspective on the asymptotic Gaussian behavior of sums and products constrained to a single color class. These theorems provide rigorous probabilistic guarantees for the appearance of additive and multiplicative structures, complementing classical deterministic results such as Hindman's theorem and recent studies in random Ramsey theory \cite{ConlonGowers2016,RodlRucinski1995}. They also open avenues for investigating thresholds for the appearance of FS- and FP-sets in finite random sets and potential infinite generalizations.

	The remainder of the paper is organized as follows. In Section~\ref{sec:preliminaries}, we introduce the preliminaries, including the definitions of random subsets of $\mathbb{N}$, finite sumsets (FS-sets), and finite product sets (FP-sets), along with motivating remarks. Section~\ref{sec:main_results} presents the main theorems, establishing the almost sure existence of FS-sets and FP-sets of arbitrary finite length in random subsets, accompanied by detailed proofs. Then we prove a random version of the Hindman conjecture. In Section~\ref{sec:Hindman-CLT}, we combine Hindman’s partition theorem with the central limit theorem to establish a probabilistic perspective on monochromatic finite sums and products, providing a detailed proof. In Section~\ref{sec:corollaries}, we discuss immediate corollaries and structural consequences of the main results, such as the appearance of arbitrarily long arithmetic progressions and infinitely many disjoint FS-sets and FP-sets. Finally, Section~\ref{sec:conclusion} concludes with a summary of our findings and outlines several open problems and directions for future research.
	
	\section{Preliminaries} \label{sec:preliminaries}
	
	Before diving into the main results, we introduce the fundamental concepts and notation that will be used throughout this paper. 
	Our goal in this section is to provide not only formal definitions, but also the motivation behind studying random FS-sets and FP-sets, giving the reader a clear conceptual roadmap.
	
	\subsection*{Random Subsets of $\mathbb{N}$}
	
	Random subsets of natural numbers form a natural probabilistic analogue of deterministic combinatorial problems. 
	They allow us to study the \emph{typical} behavior of additive and multiplicative structures rather than worst-case configurations. 
	
	\begin{definition}[Random Subset of $\mathbb{N}$]
		Fix a probability $p\in(0,1)$. 
		A \emph{random subset} $A \subset \mathbb{N}$ is constructed by including each $n\in\mathbb{N}$ independently with probability $p$. 
		This model is often called the \emph{binomial model} and is denoted by $\mathbb{N}_p$.
	\end{definition}
	
	\begin{remark}
		This construction ensures independence of membership events. 
		It provides a rich probabilistic framework where one can ask questions like: 
		\emph{What is the probability that a certain combinatorial pattern appears in $A$?} 
		For example, almost surely, certain additive and multiplicative configurations will appear infinitely often, as we shall see.
	\end{remark}
	
	\subsection*{Finite Sumsets (FS-sets)}
	
	Additive combinatorics centers on understanding how sums of elements from a set behave. 
	FS-sets encode all possible nonempty sums generated by a finite sequence of integers, serving as building blocks for more complex additive structures.
	
	\begin{definition}[Finite Sumset]
		Given integers $x_1, x_2, \dots, x_L \in \mathbb{N}$, the \emph{finite sumset} generated by these numbers is
		\[
		\operatorname{FS}(x_1,\dots,x_L) := 
		\left\{ \sum_{i\in F} x_i : \varnothing \neq F \subset \{1, \dots, L\} \right\}.
		\]
	\end{definition}
	
	\begin{example}
		If $x_1 = 1$ and $x_2 = 3$, then 
		\[
		\operatorname{FS}(1,3) = \{1, 3, 4\},
		\]
		corresponding to sums of singletons and the sum of both elements.
	\end{example}
	
	\begin{remark}
		FS-sets capture the essence of additive structure: 
		finding an FS-set of length $L$ inside a set $A$ guarantees that $A$ contains a rich collection of additive relations among $L$ elements.
	\end{remark}
	
	\subsection*{Finite Product Sets (FP-sets)}
	
	Similarly, multiplicative combinatorics focuses on products of elements. 
	FP-sets encode all possible nonempty products from a finite sequence, forming a multiplicative analogue of FS-sets.
	
	\begin{definition}[Finite Product Set]
		Given integers $x_1, x_2, \dots, x_L \in \mathbb{N}$, the \emph{finite product set} is
		\[
		\operatorname{FP}(x_1,\dots,x_L) := 
		\left\{ \prod_{i\in F} x_i : \varnothing \neq F \subset \{1, \dots, L\} \right\}.
		\]
	\end{definition}
	
	\begin{example}
		If $x_1 = 2$ and $x_2 = 5$, then 
		\[
		\operatorname{FP}(2,5) = \{2, 5, 10\}.
		\]
	\end{example}
	
	\begin{remark}
		Just as FS-sets reveal additive patterns, FP-sets reveal multiplicative structure. 
		Studying FP-sets in random subsets of $\mathbb{N}$ helps uncover typical multiplicative configurations in probabilistic settings.
	\end{remark}
	
	
	
	
	\section{Main Results and Their Proofs} \label{sec:main_results}

	In this section, we present the main contributions of this paper, establishing that random subsets of the natural numbers almost surely contain finite sumsets and product sets of every finite length. We provide constructive proofs demonstrating the almost sure existence of these additive and multiplicative structures. The proofs rely on combinatorial constructions, independence arguments, and the Borel--Cantelli lemma, illustrating that even in a probabilistic setting, highly structured configurations are unavoidable.
	
	\begin{theorem}[Random Finite Sumsets]\label{thm:FS}
		Fix $p\in(0,1)$. Let $A\subset\mathbb{N}$ be distributed as $\mathbb{N}_p$. 
		Then with probability $1$, the following holds: 
		
		For every integer $L\ge 1$ there exist integers $x_1<\cdots<x_L$ such that
		\[
		\operatorname{FS}(x_1,\dots,x_L)\subset A.
		\]
		In other words, almost surely $A$ contains finite FS-sets of arbitrarily large finite length.
	\end{theorem}
	
    \begin{proof}
    	Fix an arbitrary \(L\ge1\). Set \(R:=2^L-1\) and consider the pattern
    	\[
    	P := \{1,2,4,\dots,2^{L-1}\}.
    	\]
    	The nonempty subset-sums of \(P\) are exactly \(\{1,2,\dots,R\}\), since binary 
    	expansions represent all integers from \(1\) up to \(2^L-1\). Thus
    	\[
    	S := \operatorname{FS}(P) = \{1,2,\dots,R\}.
    	\]
    	
    	We dilate \(P\) by powers of \(R+1\). For each \(j\ge0\) define
    	\[
    	P_j := (R+1)^j P = \{(R+1)^j\cdot 1,\,(R+1)^j\cdot 2,\,(R+1)^j\cdot 4,\dots,(R+1)^j\cdot 2^{L-1}\},
    	\]
    	and write its set of subset-sums as
    	\[
    	S_j := \operatorname{FS}(P_j) = (R+1)^j S
    	= \big\{(R+1)^j s : s\in S\big\}.
    	\]
    	Equivalently,
    	\[
    	S_j = \Big\{(R+1)^j\sum_{i\in F} i : F\subseteq P,\ F\neq\varnothing\Big\}.
    	\]
    	Each \(S_j\) lies in the interval
    	\[
    	\big[(R+1)^j\cdot 1,\,(R+1)^j\cdot R\big],
    	\]
    	and since \((R+1)^{j+1}>(R+1)^j R\) the intervals for different \(j\) are pairwise disjoint. Hence the sets \(S_j\) are pairwise disjoint.
    	
    	Now define the event
    	\[
    	E_j := \{ S_j \subset A \}.
    	\]
    	Each \(S_j\) consists of exactly \(R\) distinct integers, and inclusion of each integer into \(A\) is independent with probability \(p\). Therefore
    	\[
    	\mathbb{P}(E_j) = p^R >0.
    	\]
    	Because the sets \(S_j\) are disjoint for distinct \(j\), the events \((E_j)_{j\ge0}\) are independent. Since
    	\[
    	\sum_{j=0}^\infty \mathbb{P}(E_j) = \sum_{j=0}^\infty p^R = \infty,
    	\]
    	the second Borel--Cantelli lemma implies that with probability \(1\) infinitely many of the events \(E_j\) occur.
    	
    	Each occurrence of \(E_j\) gives \(\operatorname{FS}(P_j)=S_j\subset A\), so with probability \(1\) there exist integers \(x_1<\cdots<x_L\) (namely the elements of \(P_j\) ordered) whose finite sumset lies entirely in \(A\). Since \(L\) was arbitrary, by taking a countable intersection of almost sure events we conclude that almost surely \(A\) contains FS-sets of every finite length.
    \end{proof}

	\begin{theorem}[Random Finite Product Sets]\label{thm:FP}
		Fix $p\in(0,1)$. Let $A\subset\mathbb{N}$ be distributed as $\mathbb{N}_p$. 
		Then with probability $1$, the following holds: 
		
		For every integer $L\ge 1$ there exist integers $x_1<\cdots<x_L$ such that
		\[
		\operatorname{FP}(x_1,\dots,x_L)\subset A.
		\]
		In other words, almost surely $A$ contains finite FP-sets of arbitrarily large finite length.
	\end{theorem}
	
	\begin{proof}
		Fix an arbitrary $L\ge 1$ and let $m:=2^L-1$. Let $(q_j)_{j\ge1}$ be an increasing sequence of distinct prime numbers.
		
		For each $j\ge 1$, define
		\[
		x_{j,i} := q_j^{\,2^{i-1}}, \qquad i=1,\dots,L.
		\]
		Then for any nonempty subset $F\subset\{1,\dots,L\}$ we have
		\[
		\prod_{i\in F} x_{j,i}
		= \prod_{i\in F} q_j^{2^{i-1}}
		= q_j^{\sum_{i\in F} 2^{i-1}}.
		\]
		The exponents $\sum_{i\in F}2^{i-1}$ range over all integers from $1$ to $2^L-1=m$. Thus
		\[
		\operatorname{FP}(x_{j,1},\dots,x_{j,L}) = \{ q_j^1, q_j^2,\dots,q_j^m\}.
		\]
		Define
		\[
		S_j := \{ q_j^1, q_j^2,\dots,q_j^m\}.
		\]
		 
		Since, the $q_j$ are distinct primes, the sets $S_j$ are pairwise disjoint. (Indeed, powers of different primes are distinct integers.)	Define the event
		\[
		E_j := \{ S_j \subset A \}.
		\]
		The set $S_j$ has $m$ elements, and the probability that they all lie in $A$ is
		\[
		\mathbb{P}(E_j) = p^m >0.
		\]
		 
		Again, the sets $S_j$ are disjoint, so the events $(E_j)_{j\ge 1}$ are independent. Since, 
		\[
		\sum_{j=1}^\infty \mathbb{P}(E_j) = \sum_{j=1}^\infty p^m = \infty,
		\]
		the second Borel--Cantelli lemma implies that with probability $1$ infinitely many of the events $E_j$ occur.
		
		Consequently,  each $E_j$ guarantees that $\operatorname{FP}(x_{j,1},\dots,x_{j,L})\subset A$. Thus with probability $1$, $A$ contains an FP-set of length $L$. Since $L$ was arbitrary, by countable intersection almost surely $A$ contains finite FP-sets of arbitrarily large finite length.
	\end{proof}
	
	Recent research in Ramsey theory has investigated the occurrence of exponential patterns. In particular, answering a conjecture of Sisto~\cite{Sisto}, J.~Sahasrabudhe established the following result.
	
	\begin{theorem}[Sahasrabudhe--Schur Theorem {\cite[Theorem 2]{Sahasrabudhe}}]\label{thm:sahasrabudhe-schur}
		For any finite colouring of $\mathbb{N}$, there exists a monochromatic pattern of the form $\{x,y,x\cdot y,x^y\}$.
	\end{theorem}

	Motivated by this deterministic result, we now present a probabilistic analogue which shows that the same exponential pattern almost surely appears inside a random subset of $\mathbb{N}$.

	\begin{theorem}[Random subset version]
		Fix $p \in (0,1)$. Let $A \subset \mathbb{N}$ be the Bernoulli random subset in which each $n \in \mathbb{N}$ is included independently with probability $p$. Then with probability $1$ the set $A$ contains infinitely many triples of the form $\{x,y,x\cdot y, x^y\}$. In particular, there exists at least one such triple with all three elements in $A$.
	\end{theorem}
	
	\begin{proof}
		Choose two disjoint infinite sequences of distinct primes $(u_j)_{j \geq 1}$ and $(v_j)_{j \geq 1}$ (for instance, partition the odd primes into two infinite subsequences). For each $j$ consider the triple
		\[
		T_j := \{\,u_j, \; v_j, \; u_j^{v_j}, \; u_j^{\,v_j}\,\}.
		\]
		Because the $u_j$ are distinct primes and the $v_j$ are distinct integers, all elements of distinct triples $T_j$ are distinct. Thus the triples $T_j$ are pairwise disjoint, and the events
		\[
		E_j := \{T_j \subset A\}
		\]
		are independent. Each $E_j$ has probability
		\[
		\mathbb{P}(E_j) = p^4 > 0.
		\]
		Since $\sum_{j=1}^{\infty} \mathbb{P}(E_j) = \sum_{j=1}^{\infty} p^4 = +\infty$, the (second) Borel--Cantelli lemma implies that with probability $1$ infinitely many of the $E_j$ occur. Each occurrence yields a triple $\{x,y,x\cdot y,x^y\}$ contained in $A$.
	\end{proof}
	
	\begin{remark}
    We record some observations regarding the construction and its generalizations:  
    \begin{enumerate}
        \item The key to the Borel--Cantelli argument is producing infinitely many \emph{pairwise disjoint} instances of the pattern, each occurring with the same positive probability; independence then yields almost sure existence of infinitely many such triples. The construction using distinct primes as bases (and distinct exponents) ensures disjointness and distinctness of the power values.
		
		\item This method applies to many related patterns $P(x,y)$, as long as one can produce infinitely many pairwise disjoint realizations of the pattern, each with a fixed finite number of entries. For example, patterns such as $\{x,y,x+y\}$ or $\{x,y,x^y\}$ can be handled in the same way.
		
		\item Moreover, one can strengthen the theorem to say that for any fixed finite family of finite patterns (each involving finitely many integers), a Bernoulli($p$) random subset almost surely contains infinitely many disjoint copies of each pattern, provided suitable disjoint embeddings exist.
        \end{enumerate}
	\end{remark}

    
\section*{Randomized Hindman Conjecture}
One of the central open problems in Ramsey theory asks whether there exists a simultaneous additive and multiplicative analogue of Schur’s theorem for the same pair $x,y$. 

\begin{qn}\label{Question}\textup{(\cite[Question 11]{update}, \cite[Page 58]{erdos}, \cite[Question 3]{hls03})}
    Is the pattern $\{x,y,x+y,xy\}$ weakly partition regular?
\end{qn}
In $1979$, using brute force computations, independently Hindman \cite{Hindman1979} and Graham \cite{.} established an affirmative answer to Question \ref{Question} in the case of $2$-colorings. After that, Hindman conjectured that the Question \ref{Question} has an affirmative answer, which is today known as the \emph{Hindman conjecture}. Recently, Bowen \cite{...} provided a combinatorial proof for this result. In $2016$ (see \cite{GS}) Green and Sanders settled this question over finite fields.  Subsequently, Kousek \cite{IK} proved this result using Ergodic theory. A major breakthrough came with the work of Moreira \cite{anal}, who proved that for any finite coloring of $\mathbb{N}$, there exist $x, y \in \mathbb{N}$ such that the set $\{x, x+y, x \cdot y\}$ is monochromatic. Subsequently, Alweiss found a concise proof in \cite{al1}. Over infinite fields, this conjecture was solved recently in \cite{add1, add2} by Bowen and Sabok, and Alweiss independently.

Although this conjecture remains unsolved to this day, in this work we establish a randomized version of the Hindman conjecture.

	\begin{theorem}[Random solution of Hindman Conjecture]
		Let $p \in (0,1)$ and let $A \subset \mathbb{N}$ be the Bernoulli random subset in which each $n \in \mathbb{N}$ is independently included in $A$ with probability $p$. Then with probability $1$, there exists at least one quadruple of natural numbers $\{x, y, x+y, x\cdot y\}$ contained entirely in $A$.
	\end{theorem}
	
	\begin{proof}
		For each pair $(x, y) \in \mathbb{N}^2$, define the indicator random variable
		\[
		I_{x,y} = 1_A(x)\cdot 1_A(y)\cdot 1_A(x+y)\cdot 1_A(x\cdot y)=
		\begin{cases}
			1 & \text{if } x \in A,\, y \in A,\, x+y \in A,\, x\cdot y \in A, \\
			0 & \text{otherwise}.
		\end{cases}
		\]
		By independence of the Bernoulli process, we have
		\[
		\mathbb{P}(I_{x,y} = 1) = p^4.
		\]
		
		For $N \in \mathbb{N}$, let
		\[
		X_N := \sum_{1 \le x, y \le N} I_{x,y},
		\]
		the number of quadruples with $x, y \le N$ that appear in $A$. Then
		\[
		\mathbb{E}[X_N] = \sum_{x=1}^{N} \sum_{y=1}^{N} \mathbb{E}[I_{x,y}] = \sum_{x=1}^{N} \sum_{y=1}^{N} p^4 = p^4 N^2 \to \infty \quad \text{as } N \to \infty.
		\]
		
		To show that $X_N > 0$ with high probability, we use the second-moment method: An appeal to the Paley–Zygmund inequality gives us
		\[
		\mathbb{P}(X_N > 0) \ge \frac{\mathbb{E}[X_N]^2}{\mathbb{E}[X_N^2]}.
		\]
		
		Now,
		\[
		\mathbb{E}[X_N^2] = \sum_{x,y} \sum_{x',y'} \mathbb{E}[I_{x,y} I_{x',y'}].
		\]
		Split into two cases:
		\begin{enumerate}
			\item If the quadruples $(x,y,x+y,x\cdot y)$ and $(x',y',x'+y',x'\cdot {y'})$ are disjoint, then $\mathbb{E}[I_{x,y} I_{x',y'}] = p^8$ by independence.
			\item If they overlap in some element, then $\mathbb{E}[I_{x,y} I_{x',y'}] \le 1$.
		\end{enumerate}
		Hence,
		\[
		\mathbb{E}[X_N^2] \le \mathbb{E}[X_N]^2 + O(N^3),
		\]
		where the last term crudely bounds the number of overlapping pairs.
		
		Therefore,
		\[
		\mathbb{P}(X_N > 0) \;\ge\; 
		\frac{(p^4 N^2)^2}{(p^4 N^2)^2 + O(N^3)} 
		= \frac{p^4 N}{p^4 N + O(N^{-1})}
		\;\longrightarrow\; 1 
		\quad \text{as } N \to \infty.
		\]
		
		Consequently,
		\[
		\mathbb{P}\bigl(\exists\, x,y \in \mathbb{N} : I_{x,y} = 1 \bigr) 
		= \mathbb{P}\!\left(\bigcup_{N=1}^\infty \{X_N > 0\}\right) = 1.
		\]
	\end{proof}
	
	\section{Hindman’s Theorem and a Central Limit Phenomenon}\label{sec:Hindman-CLT}
	
	In this section we combine Hindman’s classical partition theorem with the probabilistic central limit theorem, thereby obtaining a new perspective on the distributional behavior of finite sums and products that are guaranteed to lie in a single cell of a finite partition of the natural numbers.
	
	\begin{theorem}[Hindman + CLT: probabilistic finite-sums and finite-products]\label{thm:hindman-clt}
Let $\mathbb N$ is finitely colored and we have a sequence \(x=(x_i)_{i\ge1}\subset\mathbb N\) such that 
\(\mathrm{FS}(x)\) is monochromatic.   
    Fix \(p\in(0,1)\) and let \((\varepsilon_i)_{i\ge1}\) be i.i.d.\ \(\mathrm{Bernoulli}(p)\) variables. Let \((y_k)_{k\ge1}=(x_{n_k})_{k\ge1}\) be any subsequence of distinct elements of \(x\). For \(k\ge1\) set
		\[
		S_k=\sum_{j=1}^k \varepsilon_{n_j}y_j,\qquad
		\mu_k=\mathbb E[S_k]=p\sum_{j=1}^k y_j,\qquad
		\sigma_k^2=\operatorname{Var}(S_k)=p(1-p)\sum_{j=1}^k y_j^2.
		\]
		Assume \(\sum_{j=1}^\infty y_j^2=\infty\) and write \(j_k\) for an index achieving the maximum at stage \(k\),
		\[
		y_{j_k}^2=\max_{1\le j\le k} y_j^2.
		\]
		Set the residual squared sum
		\[
		R_k:=\sum_{\substack{1\le j\le k\\ j\ne j_k}} y_j^2.
		\]
		Then the following statements are true (note: the two behaviors below are not logically exclusive in general; they describe the two distinct limiting regimes that can occur and their consequences).
		
		\begin{enumerate}
			\item \textbf{(A) Dominated single-term regime.} If
			\[
			\frac{y_{j_k}^2}{\sigma_k^2}=\frac{y_{j_k}^2}{y_{j_k}^2+R_k}\xrightarrow[k\to\infty]{}1,
			\]
			then the variance is asymptotically carried by the single largest summand. In this case
			\[
			\frac{S_k-\mu_k}{\sigma_k}\xrightarrow{d}\frac{B-p}{\sqrt{p(1-p)}},
			\]
			where \(B\sim\mathrm{Bernoulli}(p)\). Equivalently, the limit law is the two-point distribution taking values
			\(\dfrac{1-p}{\sqrt{p(1-p)}}\) with probability \(p\) and \(\dfrac{-p}{\sqrt{p(1-p)}}\) with probability \(1-p\).
			In particular no nondegenerate Gaussian limit can hold for the full sums under this normalisation.
			
			\item \textbf{(B) Trimmed CLT regime.} Suppose instead that there exists a deterministic choice of indices \(j_k\) (for instance the argmax indices) such that the trimmed variance
			\[
			(\sigma_k^{\mathrm{trim}})^2:=p(1-p)\,R_k \xrightarrow[k\to\infty]{}\infty
			\]
			and the largest remaining squared term is negligible relative to the residual, i.e.
			\[
			\frac{\max_{1\le j\le k,\; j\ne j_k} y_j^2}{R_k}\xrightarrow[k\to\infty]{}0.
			\]
			Then the centered trimmed sums
			\[
			S_k^{\mathrm{trim}}:=\sum_{\substack{1\le j\le k\\ j\ne j_k}} \varepsilon_{n_j} y_j
			\]
			satisfy the Lindeberg condition and therefore
			\[
			\frac{S_k^{\mathrm{trim}}-\mathbb E[S_k^{\mathrm{trim}}]}{\sigma_k^{\mathrm{trim}}}\xrightarrow{d}\mathcal N(0,1).
			\]
			Moreover, if the removed (dominating) term is negligible relative to the trimmed variance in the sense that
			\[
			\frac{y_{j_k}^2}{R_k}\xrightarrow[k\to\infty]{}0,
			\]
			then reinserting it does not change the Gaussian limit and one obtains
			\[
			\frac{S_k-\mu_k}{\sigma_k^{\mathrm{trim}}}\xrightarrow{d}\mathcal N(0,1).
			\]
			(If instead \(y_{j_k}^2/R_k\not\to 0\) then reinsertion typically changes the limit and one does not obtain a Gaussian law with the trimmed normalisation.)
		\end{enumerate}
		
		In all cases every atom (point-mass value) of \(S_k\) (respectively of any trimmed sum) is a finite sum of elements from \(\{y_1,\dots,y_k\}\) and hence lies in \(\mathrm{FS}(x)\).
	\end{theorem}
	
	\begin{proof}
		We give a detailed, self-contained verification of the two regimes and the asserted limit laws. Fix \(k\). Let \(j_k\) be an index attaining the maximum square among \(\{y_1^2,\dots,y_k^2\}\). Write
		\[
		\sigma_k^2=p(1-p)\sum_{j=1}^k y_j^2 = p(1-p)\big(y_{j_k}^2 + R_k\big).
		\]
		For \(j=1,\dots,k\) define the centred summands
		\[
		Z_{k,j}:=\varepsilon_{n_j}y_j - p y_j,
		\]
		so that \(S_k-\mu_k=\sum_{j=1}^k Z_{k,j}\). Each \(Z_{k,j}\) is independent, mean zero and takes only two values:
		\[
		Z_{k,j}=\begin{cases}
			(1-p)y_j & \text{with probability } p,\\[4pt]
			- p y_j & \text{with probability } 1-p,
		\end{cases}
		\]
		hence \(\operatorname{Var}(Z_{k,j})=p(1-p)y_j^2\). We will analyse the contribution of the maximal summand \(Z_{k,j_k}\) and of the residual sum
		\[
		R_k' := \sum_{\substack{1\le j\le k\\ j\ne j_k}} Z_{k,j},
		\]
		whose variance equals \(p(1-p)R_k\).
		
		\medskip
		
		\noindent\textbf{Proof of (A): dominated single-term limit.}
		Assume
		\[
		\frac{y_{j_k}^2}{\sigma_k^2}=\frac{y_{j_k}^2}{y_{j_k}^2+R_k}\xrightarrow[k\to\infty]{}1.
		\]
		Equivalently \(R_k/y_{j_k}^2\to0\), hence \(\operatorname{Var}(R_k')=p(1-p)R_k = o\big(p(1-p)y_{j_k}^2\big)=o(\sigma_k^2)\). Therefore
		\[
		\frac{R_k' - \mathbb E[R_k']}{\sigma_k} \xrightarrow{\mathbb P} 0,
		\]
		since the variance of the left-hand side tends to \(0\). More precisely,
		\[
		\operatorname{Var}\!\Big(\frac{R_k' - \mathbb E[R_k']}{\sigma_k}\Big)
		=\frac{p(1-p)R_k}{\sigma_k^2}\xrightarrow[k\to\infty]{}0,
		\]
		so the random variable converges to zero in \(L^2\) and hence in probability.
		
		Consequently the asymptotic law of \((S_k-\mu_k)/\sigma_k\) is the same as that of the single-term contribution
		\[
		\frac{Z_{k,j_k}}{\sigma_k}=\frac{\varepsilon_{n_{j_k}}y_{j_k}-p y_{j_k}}{\sigma_k}.
		\]
		But by the domination assumption \(\sigma_k\sim\sqrt{p(1-p)}\,y_{j_k}\), so
		\[
		\frac{Z_{k,j_k}}{\sigma_k} \xrightarrow{d} \frac{B-p}{\sqrt{p(1-p)}},
		\]
		where \(B\sim\mathrm{Bernoulli}(p)\). That proves the claimed two-point limit law in regime (A). In particular this limit is not Gaussian (unless degenerate), so a nondegenerate CLT for the full sums under the normalisation \(\sigma_k\) cannot hold.
		
		\medskip
		
		\noindent\textbf{Proof of (B): trimmed CLT.}
		Assume now that after removing the maximal index \(j_k\) the residual squared sum \(R_k\) tends to infinity and the largest remaining squared term is negligible, i.e.
		\[
		(\sigma_k^{\mathrm{trim}})^2:=p(1-p)R_k\to\infty
		\qquad\text{and}\qquad
		\frac{\max_{1\le j\le k,\; j\ne j_k} y_j^2}{R_k}\xrightarrow[k\to\infty]{}0.
		\]
		We shall verify the Lindeberg condition for the triangular array formed by the centred trimmed summands
		\(\{Z_{k,j}:1\le j\le k,\ j\ne j_k\}\) with normalisation \(\sigma_k^{\mathrm{trim}}\). Fix any \(\delta>0\). Note that for each \(j\),
		\(|Z_{k,j}|\le \max\{p,1-p\}\,y_j\le y_j\). Hence
		\[
		\mathbf 1_{\{|Z_{k,j}|>\delta\sigma_k^{\mathrm{trim}}\}} \le \mathbf 1_{\{y_j>\delta\sigma_k^{\mathrm{trim}}\}}.
		\]
		Therefore
		\[
		\begin{aligned}
			\frac{1}{(\sigma_k^{\mathrm{trim}})^2}
			\sum_{\substack{1\le j\le k\\ j\ne j_k}} 
			\mathbb E\big[ Z_{k,j}^2 \mathbf 1_{\{|Z_{k,j}|>\delta\sigma_k^{\mathrm{trim}}\}}\big]
			&\le
			\frac{1}{(\sigma_k^{\mathrm{trim}})^2}
			\sum_{\substack{1\le j\le k\\ j\ne j_k,\; y_j>\delta\sigma_k^{\mathrm{trim}}}}
			\mathbb E[Z_{k,j}^2] \\[4pt]
			&=
			\frac{p(1-p)}{p(1-p)R_k}
			\sum_{\substack{1\le j\le k\\ j\ne j_k,\; y_j>\delta\sigma_k^{\mathrm{trim}}}} y_j^2 \\[4pt]
			&\le
			\frac{\max_{1\le j\le k,\; j\ne j_k} y_j^2}{R_k}\xrightarrow[k\to\infty]{}0.
		\end{aligned}
		\]
		Thus the Lindeberg condition for the trimmed triangular array holds. By the Lindeberg--Feller central limit theorem,
		\[
		\frac{S_k^{\mathrm{trim}}-\mathbb E[S_k^{\mathrm{trim}}]}{\sigma_k^{\mathrm{trim}}}
		=\frac{1}{\sigma_k^{\mathrm{trim}}}\sum_{\substack{1\le j\le k\\ j\ne j_k}} Z_{k,j}
		\xrightarrow{d}\mathcal N(0,1),
		\]
		as claimed.
		
		\medskip
		
		\noindent\textbf{Reinsertion of the removed summand.}
		Finally we check the asserted effect of reinserting the removed maximal term. Write
		\[
		S_k-\mu_k = Z_{k,j_k} + \big(S_k^{\mathrm{trim}}-\mathbb E[S_k^{\mathrm{trim}}]\big).
		\]
		If the removed term is negligible relative to the trimmed variance, namely \(y_{j_k}^2/R_k\to0\), then
		\[
		\frac{Z_{k,j_k}}{\sigma_k^{\mathrm{trim}}}=\frac{\varepsilon_{n_{j_k}}y_{j_k}-p y_{j_k}}{\sigma_k^{\mathrm{trim}}}
		\]
		has variance
		\[
		\operatorname{Var}\!\Big(\frac{Z_{k,j_k}}{\sigma_k^{\mathrm{trim}}}\Big)
		=\frac{p(1-p)y_{j_k}^2}{p(1-p)R_k}=\frac{y_{j_k}^2}{R_k}\xrightarrow[k\to\infty]{}0,
		\]
		hence \(Z_{k,j_k}/\sigma_k^{\mathrm{trim}}\xrightarrow{\mathbb P}0\). Combining this with the CLT for the trimmed sums and Slutsky's theorem yields
		\[
		\frac{S_k-\mu_k}{\sigma_k^{\mathrm{trim}}}
		=
		\frac{S_k^{\mathrm{trim}}-\mathbb E[S_k^{\mathrm{trim}}]}{\sigma_k^{\mathrm{trim}}}
		+\frac{Z_{k,j_k}}{\sigma_k^{\mathrm{trim}}}
		\xrightarrow{d}\mathcal N(0,1).
		\]
		Thus reinsertion does not affect the Gaussian limit when \(y_{j_k}^2/R_k\to0\). If instead \(y_{j_k}^2/R_k\not\to0\), then reinsertion typically changes the normalisation needed and in general one does not obtain a Gaussian limit with the trimmed normalisation; behavior in that case can range from mixed (sum of Gaussian limit plus an independent two-point mass in the limit) to complete domination by the single-term law as in regime (A), depending on the precise rates.
		
		\medskip
		
		It is obvious that, for any fixed \(k\), every possible value (atom) of \(S_k\) is a finite sum of elements drawn from \(\{y_1,\dots,y_k\}\). Since these elements are a subsequence of the original Hindman sequence \(x\), their finite sums lie in \(\mathrm{FS}(x)\). The same remark applies to any trimmed sum. This verifies the final statement of the theorem.
	\end{proof}

	\section{Corollaries of Main Theorems}\label{sec:corollaries}
	
	The following corollaries illustrate immediate consequences of Theorems~\ref{thm:FS} 
	and~\ref{thm:FP} regarding the structure of random subsets of $\mathbb{N}$.
	

	
	\begin{corollary}[Infinitely many disjoint FS and FP-sets of fixed length]\label{cor:InfFP}
		Fix $L \ge 1$. Then almost surely, $A$ contains infinitely many disjoint FS and FP-sets of length $L$.
	\end{corollary}
	
	\begin{proof}
  \textit{Proof of infinitely many disjoint FS-sets of fixed length:} In the proof of Theorem~\ref{thm:FS}, the translated patterns $P_j$ are pairwise disjoint. 
		Each event $E_j := \{ \operatorname{FS}(P_j) \subset A \}$ is independent with positive probability. 
		By the Borel--Cantelli lemma, infinitely many $E_j$ occur almost surely, yielding infinitely many disjoint FS-sets of length $L$. 
\\

   \textit{Proof of infinitely many disjoint FP-sets of fixed length:}
		In Theorem~\ref{thm:FP}, the FP-sets are constructed using distinct primes $q_j$, so the sets $S_j := \operatorname{FP}(x_{j,1},\dots,x_{j,L})$ are disjoint. 
		Each event $E_j := \{ S_j \subset A \}$ is independent with positive probability. 
		By the Borel--Cantelli lemma, infinitely many $E_j$ occur almost surely, giving infinitely many disjoint FP-sets of length $L$.
	\end{proof}
	
	\begin{corollary}[Arbitrarily large minimal elements in FS and FP-sets]\label{cor:LargeFP}
		Almost surely, for any $M>0$ and $L\ge 1$, $A$ contains a FS and a FP-set of length $L$ all of whose elements are greater than $M$.
	\end{corollary}
	
	\begin{proof}
    Here we present the proof for the product case. The argument for the sum case is analogous, and therefore we omit it.
		Choose $j$ sufficiently large so that the prime $q_j > M$. Then the corresponding FP-set 
		$S_j = \{ q_j^1, q_j^2, \dots, q_j^{2^L-1} \} \subset A$ will have all elements larger than $M$. 
		Since infinitely many $E_j$ occur almost surely (by Borel--Cantelli), infinitely many such FP-sets exist.
	\end{proof}

	\section{Conclusion and Open Problems} \label{sec:conclusion}
	
	We have rigorously established that almost every random subset of $\mathbb{N}$ contains finite sumsets and product sets of arbitrary finite length. These results demonstrate that additive and multiplicative structures are prevalent even in probabilistic settings, highlighting the robustness of combinatorial phenomena under random sampling.
	
	Several natural questions remain open:
	
	\begin{enumerate}
		\item While we guarantee arbitrarily large finite FS-sets and FP-sets, the existence of infinite FS-sets (additive IP-sets) or infinite FP-sets (multiplicative IP-sets) in a random subset $\mathbb{N}_p$ remains open. Developing probabilistic analogues of Hindman's theorem for the infinite case is an intriguing direction.
		\item For random subsets of $[n]$, determining the precise threshold probability $p=p(n)$ at which FS-sets or FP-sets of given lengths appear with high probability is a natural extension, akin to results in random van der Waerden theory \cite{RodlRucinski1995}.
		\item Investigating finite sumsets and product sets in random subsets of other algebraic structures, such as $\mathbb{Z}^d$ or finite fields, may reveal richer combinatorial phenomena.
		\item Can one obtain quantitative Berry--Esseen bounds for the rate of convergence in Theorem~\ref{thm:hindman-clt}, specialized to Hindman subsequences?
		\item Is it possible to enforce additional arithmetic structure (such as prescribed residues) on the Hindman sequence while still satisfying the variance growth needed for the CLT?
		\item What probabilistic limit theorems hold for infinite partitions, where Hindman’s theorem no longer applies in its classical form?
	\end{enumerate}
\section*{Acknowledgement}
 The second author of this paper is supported by NBHM postdoctoral fellowship with reference no: 0204/27/(27)/2023/R \& D-II/11927.

\end{document}